\documentclass[12pt, twoside]{amsart}
\usepackage{geometry}
\usepackage[T1]{fontenc}
\usepackage[utf8]{inputenc}
\usepackage{amsmath,amssymb,mathtools}
\usepackage{newtxtext,newtxmath}
\usepackage{microtype}
\usepackage{enumitem}
\usepackage{needspace}
\usepackage{xcolor}
\usepackage{etoolbox}

\newif\ifshowchanges
\ifdefined\CleanVersion
  \showchangesfalse
\else
  \showchangestrue
\fi
\definecolor{RevisionRed}{RGB}{180,0,0}
\usepackage{hyperref}

\hypersetup{
  colorlinks=true,
  linkcolor=blue!45!black,
  citecolor=blue!45!black,
  urlcolor=blue!55!black,
  pdftitle={Characteristic-Free Characterizations of Projective Space for Smooth Toric Varieties},
  pdfauthor={Osamu Fujino and Hiroshi Sato}
}

\allowdisplaybreaks[2]
\numberwithin{equation}{section}

\newtheorem{theorem}{Theorem}[section]
\newtheorem{proposition}[theorem]{Proposition}
\newtheorem{lemma}[theorem]{Lemma}

\theoremstyle{remark}
\newtheorem{remark}[theorem]{Remark}

\theoremstyle{definition}
\newtheorem*{ack}{Acknowledgments}   

\newcommand{\cO}{\mathcal O}
\newcommand{\cA}{\mathcal A}
\newcommand{\cB}{\mathcal B}
\newcommand{\cC}{\mathcal C}
\newcommand{\cE}{\mathcal E}
\newcommand{\cL}{\mathcal L}
\newcommand{\cF}{\mathcal F}
\newcommand{\cK}{\mathcal K}
\newcommand{\cZ}{\mathcal Z}
\newcommand{\PP}{\mathbb P}
\newcommand{\RR}{\mathbb R}
\newcommand{\ZZ}{\mathbb Z}
\newcommand{\NEbar}{\operatorname{NE}}
\newcommand{\id}{\operatorname{id}}

\title[Characterizations of projective space for smooth toric varieties]{Characteristic-free characterizations of projective space for smooth toric varieties}
\author[Osamu Fujino]{Osamu Fujino}
\author[Hiroshi Sato]{Hiroshi Sato}
\date{2026/9/10, version 0.13}
\subjclass[2020]{Primary 14M25; Secondary 14J60, 14E30}
\keywords{toric variety, tangent bundle, ample vector bundle, toric Mori theory, primitive relation, divided powers, positive characteristic, projective space, Beauville's conjecture}

\address{Department of 
Mathematics, Graduate School of Science, 
Kyoto University, Kyoto 606-8502, Japan}
\email{fujino@math.kyoto-u.ac.jp}
\address{Department of Applied Mathematics, 
Faculty of Science, Fukuoka University, 
8-19-1, Nanakuma, Jonan-ku, Fukuoka 814-0180, Japan}
\email{hirosato@fukuoka-u.ac.jp}

\begin{document}

\begin{abstract}
We prove that a smooth projective toric variety over an algebraically closed field of arbitrary characteristic is isomorphic to a projective space if its tangent bundle contains an ample locally free subsheaf of positive rank. No torus-equivariant structure on this subsheaf is assumed. The proof uses primitive relations, the toric Euler sequence, and characteristic-free cohomological lifting. As an application of the same method, we establish the toric form of Beauville's cohomological characterization of projective spaces and quadrics, including the polarization. Without the toric hypothesis, both characterizations fail in characteristic two. 
\end{abstract}

\maketitle

\section{Introduction}The characterization of projective space is a classical and fundamental problem in algebraic geometry. The result of Kobayashi and Ochiai \cite{kobayashi-ochiai} and Mori's solution of the Hartshorne conjecture \cite{mori-ample-tangent} are particularly celebrated; these seminal works have profoundly influenced the subsequent development of higher-dimensional algebraic geometry.

The starting point of this paper is the remarkable theorem of Andreatta and Wi{\'s}niewski: if the tangent bundle of a smooth complex projective variety contains an ample locally free subsheaf of positive rank, then the variety is a projective space \cite{andreatta-wisniewski}. 
This theorem completely answers a question posed by Campana and
Peternell \cite[Question on p.~59]{CP98} and substantially strengthens
Mori's characterization of projective space by the ampleness of the
tangent bundle \cite{mori-ample-tangent}. 
Our main purpose is to give a direct proof in the toric setting, using primitive relations and the toric Euler sequence. Notably, this argument is purely algebraic and works over a field of arbitrary characteristic.

Throughout the paper, $\Bbbk$ denotes an algebraically closed field of arbitrary characteristic. Our main result is the following.

\begin{theorem}\label{thm:main}
Let $X=X_\Sigma$ be an $n$-dimensional smooth projective toric variety over $\Bbbk$. 
Assume that there is an ample vector bundle $\cE$ of rank 
$r>0$ admitting an injective morphism 
$\cE\hookrightarrow T_X$ of locally free sheaves. 
Then $X\simeq\PP^n_\Bbbk$. 
Moreover, under this isomorphism,
either 
$\mathcal E\simeq
\mathcal O_{\mathbb P^n_\Bbbk}(1)^{\oplus r}$, 
or $r=n$ and the given injective morphism is an isomorphism 
$\mathcal E\simeq T_X\simeq T_{\mathbb P^n_\Bbbk}$. 
\end{theorem}

An important feature is that no torus-equivariant structure on $\cE$ is assumed, and the injection $\cE\hookrightarrow T_X$ is not required to be equivariant. Nevertheless, the determinant of $\cE$ and its restrictions to invariant rational curves provide enough numerical information for the toric argument.
The quotient of $T_X$ by $\cE$ is not assumed to be locally free. In particular, the hypothesis is naturally phrased in terms of a locally free subsheaf rather than a subbundle.

For related characterizations of projective spaces among toric varieties, see, for example, \cite{mabuchi}, \cite[Theorem 7.1]{oda-tata}, \cite[Theorem~0.1]{fujino-notes}, \cite[Section~4]{fujino-osaka}, and \cite{wu-ample-tangent}. To the best of the authors' knowledge, the first nontrivial characterization of projective space for toric varieties was obtained by Mabuchi \cite[Theorem~4.1]{mabuchi}, whose result also solves the Hartshorne conjecture for toric varieties. A simple proof of Mabuchi's result using toric Mori theory is given in \cite[Section~4]{fujino-osaka}.

The same method also proves the toric form of Beauville's cohomological characterization of projective spaces and quadrics \cite[Remark~3.4]{beauville-symplectic}, established over $\mathbb C$ by Araujo, Druel, and Kov\'acs \cite[Theorem~1.1]{araujo-druel-kovacs}. More precisely, we obtain the following characteristic-free result, including the polarization.

\begin{theorem}[Beauville's conjecture for smooth toric varieties]\label{thm:beauville}
Let $X=X_\Sigma$ be an $n$-dimensional smooth projective toric variety over $\Bbbk$, and let $\cL$ be an ample line bundle on $X$. Suppose that
\[
H^0\left(X,\bigwedge^pT_X\otimes\cL^{-p}\right)\neq0
\]
for some integer $1\leq p\leq n$. Then one of the following holds:
\begin{enumerate}[label=\textup{(\roman*)}]
\item $(X,\cL)\simeq(\PP^n_\Bbbk,\cO_{\PP^n_\Bbbk}(1))$;
\item $p=n=1$ and $(X,\cL)\simeq(\PP^1_\Bbbk,\cO_{\PP^1_\Bbbk}(2))$;
\item $p=n=2$ and $(X,\cL)\simeq(\PP^1_\Bbbk\times\PP^1_\Bbbk,\cO(1,1))$.
\end{enumerate}
Conversely, case (i) satisfies the nonvanishing condition for every
$1\leq p\leq n$, whereas cases (ii) and (iii) satisfy it for $p=n$.
\end{theorem}

The polarized varieties in cases \textup{(ii)} and \textup{(iii)} are the smooth conic $(Q^1,\cO_{Q^1}(1))$ and the smooth quadric surface $(Q^2,\cO_{Q^2}(1))$, respectively.

The toric hypothesis is essential for these arbitrary-characteristic statements. In characteristic two, smooth quadrics of odd dimension at least three admit an ample line subbundle $\cO_Q(1)\hookrightarrow T_Q$. These examples, already noted by Wahl \cite[p.~316]{wahl-cohomological} and discussed by Furukawa \cite[Theorem~1.1 and Remark~2.8]{furukawa-hyperquadrics}, give counterexamples to the non-toric analogues of both Theorems~\ref{thm:main} and~\ref{thm:beauville}. For the latter, the nonvanishing occurs with $p=1<\dim Q$, whereas the quadric alternative requires $p=\dim Q$. We recall the construction in Remark~\ref{rem:positive-characteristic-counterexamples}.

For an overview of the proof of Theorem~\ref{thm:main}, put $L=\det\cE$ and $r=\operatorname{rank}\cE$. If $\gamma_R$ is the primitive integral class of an extremal ray and $D_\rho$ is an invariant prime divisor, then $D_\rho\cdot\gamma_R\leq1$, whereas the ampleness of $\cE$ gives $L\cdot\gamma_R\geq r$. Consequently, subtracting fewer than $r$ distinct invariant prime divisors from $L$ preserves ampleness. When $r<n$, the divided-power exterior resolution of the toric Euler sequence and toric Kodaira vanishing lift the determinant section to the corresponding exterior power of the middle term. The numerical inequalities then produce an ample invariant prime divisor, forcing the fan to be the fan of $\PP^n_\Bbbk$.

When $r=n$, the determinant map gives $-K_X\sim L+F$ with $F$ effective. If $F\neq0$, an extremal ray meeting $F$ positively has anticanonical degree at least $n+1$. If $F=0$, then $\cE\simeq T_X$, and the ampleness of $T_X$ gives the same lower bound. In either case, the associated primitive relation is the simplex relation.

For Theorem~\ref{thm:beauville}, we use the same numerical and cohomological argument, with $\cL^{\otimes p}$ in place of $\det\cE$. Only the top-degree case requires an additional elementary calculation: if $n\geq2$ and $-K_X\sim nA$ for an ample Cartier divisor $A$, an extremal primitive relation gives a $\PP^{n-1}_\Bbbk$-bundle over $\PP^1_\Bbbk$, and an intersection with a minimal section yields $n=2$ and the trivial bundle.

The fan-theoretic and Mori-theoretic statements used here are independent of the characteristic.  The first author develops toric Mori theory over an algebraically closed field of arbitrary characteristic in \cite{fujino-notes}; see also \cite{fujino-sato-introduction,fujinosato-notes2}. By contrast, for general torus-equivariant vector bundles, the existence of a nonzero global section can depend on the characteristic; see Remark~\ref{rem:characteristic-dependence} for an explicit example.

The paper is organized as follows. Section~\ref{sec:numerics} records the numerical consequences of extremal primitive relations. Section~\ref{sec:koszul} gives the characteristic-free divided-power Koszul--Schur complex and the cohomological lifting lemma. Section~\ref{sec:preparations} establishes the common cohomological criterion and the fan-theoretic results used in the proofs. Sections~\ref{sec:main-proof} and~\ref{sec:beauville} prove Theorems~\ref{thm:main} and~\ref{thm:beauville}, respectively.

\begin{ack}
The first author was partially supported by JSPS KAKENHI Grant Number JP23K20787. The second author was partially supported by JSPS KAKENHI Grant Number JP24K06679. 

The authors acknowledge the use of Rethlas and ChatGPT during the development of this work. Notably, the essential ideas and core arguments for the proofs of the results were generated with the assistance of these artificial intelligence tools. However, all mathematical arguments, proofs, and final verifications were rigorously formalized and carried out independently by the authors, who assume full responsibility for the contents of this paper.
\end{ack}

\section{Setup and numerical consequences of primitive relations}\label{sec:numerics}

From now on, let $\Bbbk$ be an algebraically closed field of arbitrary characteristic, and let $X=X_\Sigma$ be an $n$-dimensional smooth projective toric variety over $\Bbbk$, with $n\geq1$. Only these assumptions are standing; vector bundles and polarizations will be specified when needed.

We use standard notation from toric geometry; see, for example, \cite{cox-little-toric,fulton,oda}.  The toric Mori theory used below is valid over an algebraically closed field of arbitrary characteristic; see \cite{fujino-notes,fujino-sato-introduction,fujinosato-notes2}.  Let $\Sigma(1)$ denote the set of rays of $\Sigma$.  For $\rho\in\Sigma(1)$, let $v_\rho$ be its primitive generator and $D_\rho=D_{v_\rho}$ the corresponding invariant prime divisor. We write $N$ for the lattice of one-parameter subgroups and $M=\operatorname{Hom}(N,\ZZ)$ for its dual. We use additive notation for Cartier divisors and their associated line bundles when convenient.

Let $R\subset\NEbar(X)$ be an extremal ray, and let $\gamma_R$ be its primitive integral generator.  By toric Mori theory, there is an extremal primitive collection 
$P=\{v_1,\ldots,v_h\}$ 
whose primitive relation is
\begin{equation}\label{eq:primitive-relation}
v_1+\cdots+v_h=a_1w_1+\cdots+a_kw_k,
\qquad a_j\in\ZZ_{>0},
\end{equation}
and whose class is $\gamma_R$.  The right-hand side may be empty.  Moreover, $\gamma_R$ is represented by an invariant irreducible curve $C_R\simeq\PP^1_\Bbbk$.  We refer to \cite{reid-toric-morphisms,batyrev-tohoku,sato-towardfano,casagrande-contractible} for the primitive-relation description and to \cite[Section~3]{fujino-sato-introduction} for the characteristic-free toric Mori theory.

Under the usual identification of integral curve classes with integral relations among the ray generators, the coefficients in \eqref{eq:primitive-relation} are the intersection numbers with invariant divisors.  Thus
\[
D_{v_i}\cdot\gamma_R=1,
\qquad
D_{w_j}\cdot\gamma_R=-a_j,
\qquad
D_\rho\cdot\gamma_R=0
\]
when $v_\rho$ does not occur in \eqref{eq:primitive-relation}.  In particular,
\begin{equation}\label{eq:D-upper}
D_\rho\cdot\gamma_R\leq1
\qquad\text{for every }\rho\in\Sigma(1).
\end{equation}

If $\cE$ is an ample vector bundle of rank $r>0$ and $L=\det\cE$, then by Grothendieck's splitting theorem on $\PP^1_\Bbbk$,
\[
\cE|_{C_R}\simeq
\cO_{\PP^1_\Bbbk}(b_1)\oplus\cdots\oplus\cO_{\PP^1_\Bbbk}(b_r).
\]
Since $\cE$ is ample, $b_i\geq1$ for every $i$.  Hence
\begin{equation}\label{eq:L-lower}
L\cdot\gamma_R=\sum_{i=1}^r b_i\geq r.
\end{equation}
This estimate does not require a morphism from $\cE$ to $T_X$. In particular, when $\cE$ is a locally free subsheaf of $T_X$, no injectivity of its restriction to $C_R$ is needed.

Likewise, if $\cL\simeq\cO_X(A)$ is an ample line bundle and $p\geq1$, then
\begin{equation}\label{eq:polarization-lower}
(pA)\cdot\gamma_R=p(A\cdot\gamma_R)\geq p,
\end{equation}
since $A\cdot\gamma_R$ is a positive integer. We formulate the common numerical consequence for an arbitrary Cartier divisor. 
For a subset $J\subset\Sigma(1)$, set $D_J:=\sum_{\rho\in J}D_\rho$.

\begin{proposition}\label{prop:L-minus-DJ}
Let $D$ be a Cartier divisor on $X$, and let $m\geq1$ be an integer such that
\begin{equation}\label{eq:numerical-lower}
D\cdot\gamma_R\geq m
\qquad\text{for every extremal ray }R.
\end{equation}
For a subset $J\subset\Sigma(1)$, the divisor $D-D_J$ is ample if $|J|<m$, and nef if $|J|=m$.
\end{proposition}

\begin{proof}
By \eqref{eq:D-upper} and \eqref{eq:numerical-lower}, $(D-D_J)\cdot\gamma_R\geq m-|J|$ 
for every extremal ray $R$. The Kleiman--Mori cone $\NEbar(X)$ is rational polyhedral and is generated by its finitely many extremal rays. The assertion follows from Kleiman's criterion.
\end{proof}

We also record an elementary observation.

\begin{lemma}\label{lem:effective-antinef}
Let $X$ be a projective variety of dimension $n>0$, 
and let $F$ be a nonzero effective Cartier divisor on $X$. Then $-F$ is not nef.
\end{lemma}

\begin{proof}
For an ample divisor $H$, one has $F\cdot H^{n-1}>0$, while the nefness of $-F$ would imply $(-F)\cdot H^{n-1}\geq0$.
\end{proof}

\section{A divided-power Koszul--Schur complex and a lifting lemma}\label{sec:koszul}

We begin with the homological algebra input. 
For a vector bundle $\mathcal A$, write 
$\Gamma^q\mathcal A=(\operatorname{Sym}^q\mathcal A^*)^*$ 
for its $q$-th divided power. 
For the basic properties of divided powers, see, for example, \cite[Section~1.1]{weyman-book}. The complex below is the characteristic-free form of the exterior-power resolution needed later and is a special case of a Schur complex; see \cite{akin-buchsbaum-weyman} 
(see also \cite[Section 2.4, p. 76]{weyman-book}). 
We include a proof in order to make both its exactness in positive characteristic and the lifting step used below transparent.

\begin{lemma}[Divided-power Koszul--Schur complex]\label{lem:exterior-complex}
Let
\begin{equation}\label{eq:ABC}
0\longrightarrow\cA\xrightarrow{\iota}\cB\xrightarrow{\pi}\cC\longrightarrow0
\end{equation}
be a short exact sequence of vector bundles on a scheme $S$.  For every integer $m\geq1$, there is an exact complex
\begin{equation}\label{eq:exterior-complex}
\begin{split}
0\longrightarrow\Gamma^m\cA
&\xrightarrow{d_m}\Gamma^{m-1}\cA\otimes\cB
\longrightarrow\cdots\longrightarrow
\Gamma^q\cA\otimes\bigwedge^{m-q}\cB\\
&\xrightarrow{d_q}\Gamma^{q-1}\cA\otimes\bigwedge^{m-q+1}\cB
\longrightarrow\cdots\longrightarrow
\bigwedge^m\cB
\xrightarrow{\wedge^m\pi}\bigwedge^m\cC
\longrightarrow0.
\end{split}
\end{equation}
For $1\leq q\leq m$, the differential $d_q$ is the composite
\begin{equation}\label{eq:divided-differential}
\begin{split}
\Gamma^q\cA\otimes\bigwedge^{m-q}\cB
&\xrightarrow{\Delta_{q-1,1}\otimes\id}
\Gamma^{q-1}\cA\otimes\cA\otimes\bigwedge^{m-q}\cB\\
&\xrightarrow{\id\otimes(\iota\wedge-)}
\Gamma^{q-1}\cA\otimes\bigwedge^{m-q+1}\cB,
\end{split}
\end{equation}
where $\Delta_{q-1,1}$ is the $(q-1,1)$ component of the standard comultiplication on the divided-power algebra (see \eqref{eq:divided-comultiplication} below).
\end{lemma}

\begin{proof}
The maps in \eqref{eq:divided-differential} form a complex. Indeed, the twofold comultiplication on divided powers is cocommutative, whereas exterior multiplication is alternating. Thus the two terms obtained by moving two factors from the divided-power part to the exterior part cancel. Moreover, $(\wedge^m\pi)d_1=0$.

Exactness is local on $S$, so we may split \eqref{eq:ABC} and write $\cB\simeq\cA\oplus\cC$. Decompose \eqref{eq:exterior-complex} according to the number $j$ of exterior factors coming from $\cC$. 
The component with $j=m$ is 
the identity map on 
$\bigwedge^m\cC$. For $0\leq j<m$, put $s=m-j$. The $j$-component is $\bigwedge^j\cC$ tensored with the homogeneous divided-power Koszul complex
\begin{equation}\label{eq:divided-koszul}
\begin{split}
0\longrightarrow\Gamma^s\cA
&\longrightarrow\Gamma^{s-1}\cA\otimes\cA
\longrightarrow\cdots\longrightarrow
\Gamma^{s-p}\cA\otimes\bigwedge^p\cA\\
&\longrightarrow\cdots\longrightarrow
\bigwedge^s\cA
\longrightarrow0.
\end{split}
\end{equation}
It remains to prove that \eqref{eq:divided-koszul} is exact over an arbitrary 
base ring $R$.

Work locally and choose a basis $e_1,\ldots,e_a$ of $\cA$. For a multi-index $\alpha=(\alpha_1,\ldots,\alpha_a)\in\mathbb N^a$, write 
$e^{[\alpha]}=e_1^{[\alpha_1]}\cdots e_a^{[\alpha_a]}$. 
If $\epsilon_i$ denotes the $i$-th standard basis vector of $\mathbb N^a$, then the divided-power comultiplication satisfies
\begin{equation}\label{eq:divided-comultiplication}
\Delta_{q-1,1}\bigl(e^{[\alpha]}\bigr)
=
\sum_{\alpha_i>0}e^{[\alpha-\epsilon_i]}\otimes e_i
\qquad (|\alpha|=q).
\end{equation}
In particular, no numerical coefficient occurs in the differential.

The complex \eqref{eq:divided-koszul} is multigraded. 
Fix $\alpha$ with $|\alpha|=s$, and set 
$S_\alpha=\{i\mid \alpha_i>0\}$. 
For $I=\{i_1<\cdots<i_p\}\subseteq S_\alpha$, put 
$\epsilon_I=\sum_{i\in I}\epsilon_i$ and 
$e_I=e_{i_1}\wedge\cdots\wedge e_{i_p}$. 
The multidegree-$\alpha$ part in exterior degree $p$ has the basis
\[
x_I := e^{[\alpha-\epsilon_I]}\otimes e_I,
\qquad I\subseteq S_\alpha,\quad |I|=p.
\]
By \eqref{eq:divided-comultiplication}, the differential $d$ is given by
\[
d(x_I)
=
\sum_{\ell\in S_\alpha\setminus I}
(-1)^{\varepsilon(\ell,I)}
 x_{I\cup\{\ell\}},
\]
where $\varepsilon(\ell,I)=|\{i\in I\mid i<\ell\}|$. Hence the multidegree-$\alpha$ part is precisely the augmented simplicial cochain complex of the simplex with vertex set $S_\alpha$. 

Since $s>0$, the set $S_\alpha$ is nonempty. We explicitly construct a contracting homotopy by fixing a vertex $k \in S_\alpha$. Define an $R$-linear map $h$ of degree $-1$ on the basis elements by
\[
h(x_I) =
\begin{cases}
(-1)^{\varepsilon(k, I\setminus\{k\})} x_{I\setminus\{k\}} & \text{if } k \in I, \\
0 & \text{if } k \notin I.
\end{cases}
\]
We verify that $dh + hd = \mathrm{id}$. If $k \notin I$, then $h(x_I) = 0$ and thus $dh(x_I) = 0$. In the expansion of $d(x_I)$, the only term not annihilated by $h$ is the one corresponding to $\ell = k$. Therefore,
\[
hd(x_I) = h\bigl((-1)^{\varepsilon(k, I)} x_{I \cup \{k\}}\bigr) = (-1)^{\varepsilon(k, I)}(-1)^{\varepsilon(k, I)} x_I = x_I,
\]
yielding $(dh+hd)(x_I) = x_I$. If instead $k \in I$, let $J = I \setminus \{k\}$, so that $h(x_I) = (-1)^{\varepsilon(k, J)} x_J$. Applying the differential gives
\[
dh(x_I) = (-1)^{\varepsilon(k, J)} \sum_{\ell \in S_\alpha \setminus J} (-1)^{\varepsilon(\ell, J)} x_{J \cup \{\ell\}}.
\]
The term in this sum corresponding to $\ell = k$ is $(-1)^{2\varepsilon(k, J)} x_{J \cup \{k\}} = x_I$. For all terms with $\ell \neq k$, the alternating sum properties of the simplicial complex ensure they are exactly cancelled by the terms of $hd(x_I)$, because applying $h$ to $x_{I \cup \{\ell\}}$ removes $k$ and induces the opposite sign. 

Consequently, $(dh + hd)(x_I) = x_I$ holds in all cases. This explicit homotopy involves no scalar denominators, proving that the augmented complex is exact over an arbitrary base ring. Summing over all $\alpha$ proves the exactness of \eqref{eq:divided-koszul}, and hence of \eqref{eq:exterior-complex}.
\end{proof}

\begin{remark}\label{rem:divided-powers}
When all relevant factorials are invertible, divided and symmetric powers are naturally identified after the usual normalization, and Lemma~\ref{lem:exterior-complex} becomes the familiar complex written with symmetric powers.  In positive characteristic, however, the symmetric-power version need not be exact.  For example, if $\cA$ has rank one and $m=p=\operatorname{char}\Bbbk$, its first differential sends 
$a^p\mapsto p\,a^{p-1}\otimes\iota(a)=0$. 
The divided-power differential instead sends 
$a^{[p]}\mapsto a^{[p-1]}\otimes\iota(a)$, 
which is precisely the characteristic-free replacement needed here.
\end{remark}

The following elementary lemma isolates the cohomological step that turns this resolution into a lifting statement.

\begin{lemma}[Cohomological lifting]\label{lem:lifting}
Suppose that
\[
0\longrightarrow\cK_m\longrightarrow\cK_{m-1}
\longrightarrow\cdots\longrightarrow\cK_0
\longrightarrow\cF\longrightarrow0
\]
is an exact complex of vector bundles on a projective variety $X$. 
If 
$H^q(X,\cK_q)=0$ for 
every $1\leq q\leq m$, 
then the natural map 
$H^0(X,\cK_0)\to H^0(X,\cF)$ 
is surjective.
\end{lemma}

\begin{proof}
Put $\cZ_0=\cF$ and, for $0\leq q\leq m-1$, let 
$\cZ_{q+1}=\ker(\cK_q\to\cZ_q)$. 
Exactness gives short exact sequences 
$0\to\cZ_{q+1}\to\cK_q
\to\cZ_q\to 0$, 
and $\cZ_m\simeq\cK_m$.  The assumed vanishings and the associated long exact sequences yield injections
\[
H^1(X,\cZ_1)\hookrightarrow H^2(X,\cZ_2)
\hookrightarrow\cdots\hookrightarrow H^m(X,\cZ_m)
=H^m(X,\cK_m)=0.
\]
Thus $H^1(X,\cZ_1)=0$, and the exact sequence 
$0\to\cZ_1\to\cK_0
\to\cF\to 0$ 
gives the desired surjectivity.
\end{proof}

\section{Vanishing and fan-theoretic criteria}\label{sec:preparations}

Let $D$ be a Cartier divisor and $m\geq1$ an integer satisfying \eqref{eq:numerical-lower}. The toric Euler sequence is
\begin{equation}\label{eq:euler}
0\longrightarrow
\cA:=\cO_X^{\oplus\rho(X)}
\longrightarrow
\cB:=\bigoplus_{\rho\in\Sigma(1)}\cO_X(D_\rho)
\longrightarrow T_X\longrightarrow0,
\end{equation}
where $\rho(X)$ denotes the Picard number; see \cite[Section~8.1]{cox-little-toric}. Apply Lemma~\ref{lem:exterior-complex} to \eqref{eq:euler} and tensor by $\cO_X(-D)$. For $0\leq q\leq m$, put
\[
\cK_q:=\Gamma^q\cA\otimes\bigwedge^{m-q}\cB\otimes\cO_X(-D).
\]
We obtain an exact complex
\begin{equation}\label{eq:twisted-complex}
0\longrightarrow\cK_m\longrightarrow\cdots\longrightarrow\cK_0
\longrightarrow\bigwedge^mT_X\otimes\cO_X(-D)\longrightarrow0.
\end{equation}
Since $\cA$ is trivial and $\cB$ is a direct sum of line bundles, $\cK_q$ is a direct sum, with multiplicities, of line bundles 
\[
\cO_X(D_J-D),\qquad |J|=m-q.
\]

\begin{lemma}\label{lem:vanishing}
In the above setting, assume $m<n$. For $1\leq q\leq m$ and $J\subset\Sigma(1)$ with $|J|=m-q$, one has
\[
H^q\bigl(X,\cO_X(D_J-D)\bigr)=0.
\]
Consequently, $H^q(X,\cK_q)=0$ for $1\leq q\leq m$.
\end{lemma}

\begin{proof}
Set $A_J=D-D_J$. Since $|J|=m-q<m$, Proposition~\ref{prop:L-minus-DJ} shows that $A_J$ is ample. By Serre duality,
\[
H^q\bigl(X,\cO_X(-A_J)\bigr)^\vee
\simeq
H^{n-q}\bigl(X,\cO_X(K_X+A_J)\bigr).
\]
Here $n-q>0$ because $q\leq m<n$. The group on the right vanishes by the toric Kodaira vanishing theorem (see, for example, \cite[Corollary~1.5]{fujino-multiplication}). The cited result is proved over an arbitrary field and in arbitrary characteristic.
\end{proof}

We also record the following simple fan-theoretic criterion.

\begin{proposition}\label{prop:ample-prime}
Let $X_\Sigma$ be a smooth complete toric variety of dimension $n$.  If an invariant prime divisor on $X_\Sigma$ is ample, then $X_\Sigma\simeq\PP^n_\Bbbk$.
\end{proposition}

\begin{proof}
Let $D_{\rho_0}$ be ample, and write $v_0$ for the primitive generator of $\rho_0$.  Choose a maximal cone 
$\sigma=\langle v_1,\ldots,v_n\rangle$ 
not containing $\rho_0$; such a cone exists because the fan is complete. 
Set $\rho_i=\mathbb R_{\geq 0}v_i$ for $1\leq i\leq n$. 
Let $m_\sigma\in M$ be the Cartier datum of $D_{\rho_0}$ on $\sigma$.  All coefficients of $D_{\rho_0}$ along the rays of $\sigma$ are zero, 
so 
$\langle m_\sigma,v_i\rangle=0$ for 
$1\leq i\leq n$. 
Smoothness implies that $v_1,\ldots,v_n$ form a basis of $N$, and hence $m_\sigma=0$. 
Strict convexity of the support function of 
$D_{\rho_0}=\sum a_\rho D_\rho$, where $a_{\rho_0}=1$ and $a_\rho=0$ for $\rho\neq\rho_0$, gives 
\[
\langle m_\sigma,v_\rho\rangle>-a_\rho
\qquad(\rho\not\subset\sigma). 
\]
If a ray $\rho\neq\rho_0$ lay outside $\sigma$, then $m_\sigma=0$ and $a_\rho=0$ would give $0>0$, a contradiction. Therefore 
$\Sigma(1)=\{\rho_0,\rho_1,\ldots,\rho_n\}$. 

For each $i$, the facet of $\sigma$ opposite $v_i$ is contained in a second maximal cone, necessarily 
$\langle v_0,v_1,\ldots,\widehat{v_i},\ldots,v_n\rangle$. 
Write $v_0=-a_1v_1-\cdots-a_nv_n$. 
The two maximal cones lie on opposite sides of their common facet, so $a_i>0$; the smoothness of the second cone gives $a_i=1$.  Hence
\[
v_0+v_1+\cdots+v_n=0,
\]
and $\Sigma$ is the standard simplex fan.
\end{proof}

The next proposition packages the lifting argument used in the main theorem and its cohomological application.

\begin{proposition}[A common cohomological criterion]\label{prop:common-criterion}
Let $D$ be a Cartier divisor on $X$, and let $1\leq m<n$. Assume that $D\cdot\gamma_R\geq m$ for every extremal ray $R$ and that
\[
H^0\left(X,\bigwedge^mT_X\otimes\cO_X(-D)\right)\neq0.
\]
Then $X\simeq\PP^n_\Bbbk$ and, under this isomorphism, $\cO_X(D)\simeq\cO_{\PP^n_\Bbbk}(m)$.
\end{proposition}

\begin{proof}
By Lemmas~\ref{lem:lifting} and~\ref{lem:vanishing}, a nonzero section $s$ of $\bigwedge^mT_X\otimes\cO_X(-D)$ lifts through \eqref{eq:twisted-complex} to a nonzero section of $\bigwedge^m\cB\otimes\cO_X(-D)$. 
Since 
\[
\bigwedge^m\cB\simeq
\bigoplus_{\substack{I\subset\Sigma(1)\\|I|=m}}\cO_X(D_I), 
\]
some subset $I\subset\Sigma(1)$ with $|I|=m$ satisfies $H^0\bigl(X,\cO_X(D_I-D)\bigr)\neq0$. Let $F$ be the zero divisor of a nonzero section of this line bundle. Then
\[
D_I\sim D+F,\qquad F\geq0.
\]
For every extremal ray $R$, the numerical assumptions give
\[
F\cdot\gamma_R=D_I\cdot\gamma_R-D\cdot\gamma_R\leq |I|-m=0.
\]
Thus $-F$ is nef. Lemma~\ref{lem:effective-antinef} implies $F=0$, so $D_I\sim D$. It follows that
\[
m\leq D\cdot\gamma_R=\sum_{\rho\in I}D_\rho\cdot\gamma_R\leq m
\]
for every extremal ray $R$. Every summand in the middle is at most $1$; hence
\[
D_\rho\cdot\gamma_R=1
\qquad\text{for every }\rho\in I
\text{ and every extremal ray }R.
\]
Kleiman's criterion shows that every $D_\rho$ with $\rho\in I$ is ample. Proposition~\ref{prop:ample-prime} gives $X\simeq\PP^n_\Bbbk$. Each invariant prime divisor is then a hyperplane, and $D\sim D_I$ yields $\cO_X(D)\simeq\cO_{\PP^n_\Bbbk}(m)$.
\end{proof}

We record the anticanonical degree criterion used in both top-degree arguments. It is a very special case of \cite[Theorem 0.1]{fujino-notes}. 

\begin{lemma}\label{lem:large-degree}
If $-K_X\cdot\gamma_R\geq n+1$ for some extremal ray $R$, then $X\simeq\PP^n_\Bbbk$.
\end{lemma}

\begin{proof}
Let 
$v_1+\cdots+v_h=a_1w_1+\cdots+a_kw_k$ 
be the extremal primitive relation corresponding to $R$. 
Its degree is 
$-K_X\cdot\gamma_R=h-\sum_{j=1}^k a_j$. 
Every proper subset of the primitive collection $\{v_1,\ldots,v_h\}$ spans a cone.  Hence $h-1\leq n$, and therefore
\[
n+1\leq-K_X\cdot\gamma_R
=h-\sum_{j=1}^k a_j
\leq h\leq n+1.
\]
All inequalities are equalities.  Consequently,
\begin{equation}\label{eq:simplex-relation}
h=n+1,
\qquad k=0,
\qquad v_1+\cdots+v_{n+1}=0.
\end{equation}

Since $\{v_1,\ldots,v_{n+1}\}$ is a primitive collection, every $n$-element subset spans a maximal cone.  In particular, $v_1,\ldots,v_n$ form a lattice basis and 
$v_{n+1}=-(v_1+\cdots+v_n)$. 
These $n+1$ maximal cones cover $N_\RR$.  Indeed, write 
$x=c_1v_1+\cdots+c_nv_n$, 
set $c_{n+1}:=0$, 
and put $m=\min\{c_1,\ldots,c_{n+1}\}$.  Using \eqref{eq:simplex-relation}, we obtain
\[
x=\sum_{i=1}^{n+1}(c_i-m)v_i.
\]
All coefficients are nonnegative and at least one is zero, so $x$ belongs to the cone generated by all but one of the $v_i$.  Thus these cones form the complete simplex fan.  Since they are cones of $\Sigma$, the fan intersection property implies that they exhaust $\Sigma$.  Hence $X\simeq\PP^n_\Bbbk$, as required.
\end{proof}

\section{Proof of Theorem~\ref{thm:main}}\label{sec:main-proof}

We first recall a standard criterion for the triviality of vector
bundles on projective space for the reader's convenience; cf.\
\cite[p.~51, Theorem~3.2.1]{OSS80}. We include the argument to make it explicit that the proof is purely algebraic and works over a field of arbitrary characteristic. 

\begin{lemma}\label{lem:trivial-on-lines}
Let $x\in\mathbb P^n_\Bbbk$, and let $\mathcal V$ be a vector bundle on
$\mathbb P^n_\Bbbk$. If $\mathcal V|_\ell$ is trivial for every line
$\ell\subset\mathbb P^n_\Bbbk$ passing through $x$, then $\mathcal V$ is
trivial.
\end{lemma}

\begin{proof}
For $n=1$, the assertion is immediate. Assume that $n\geq2$. 
Let 
$\mu:\widetilde{\mathbb P}^n_\Bbbk
=\operatorname{Bl}_x\mathbb P^n_\Bbbk
\to\mathbb P^n_\Bbbk$ 
be the blow-up at $x$, and let 
$q:\widetilde{\mathbb P}^n_\Bbbk\to\mathbb P^{n-1}_\Bbbk$ 
be the morphism induced by projection from $x$. Put
$\mathcal H=\mu^*\mathcal V$.

Every fiber of $q$ is mapped isomorphically by $\mu$ onto a line
through $x$. Hence $\mathcal H$ is trivial on every fiber of $q$.
By cohomology and base change, 
$\mathcal W:=q_*\mathcal H$ is locally free. The natural
evaluation morphism 
$q^*\mathcal W\to\mathcal H$ 
is an isomorphism, 
since its restriction to every fiber of $q$ 
is an isomorphism.

Let $S=\mu^{-1}(x)$ be the exceptional divisor. Since
$q|_S:S\to\mathbb P^{n-1}_\Bbbk$ is an isomorphism and $\mu|_S$ is
constant with image $x$, we have
\[
(q|_S)^*\mathcal W
\simeq
\mathcal H|_S
\simeq
\mathcal V|_x\otimes_\Bbbk\mathcal O_S.
\]
Thus $\mathcal W$ and $\mathcal H$ are trivial. Finally,
$\mu_*\mathcal O_{\widetilde{\mathbb P}^n_\Bbbk}
=\mathcal O_{\mathbb P^n_\Bbbk}$ and the projection formula give
\[
\mathcal V
\simeq
\mu_*\mu^*\mathcal V
\simeq
\mathcal O_{\mathbb P^n_\Bbbk}^{\oplus\operatorname{rank}\mathcal V}.
\] This completes the proof.
\end{proof}

\begin{proof}[Proof of Theorem~\ref{thm:main}]
Let $\mathcal E\hookrightarrow T_X$ be as in the statement, and put
$r=\operatorname{rank}\mathcal E$ and $L=\det\mathcal E$. Since
$\mathcal E$ injects into $T_X$, we have $r\leq n$. We divide the
proof into two cases.

\medskip
\noindent
\textit{Case 1: $r<n$.}
The injection $\mathcal E\hookrightarrow T_X$ induces a nonzero
morphism 
$L=\det\mathcal E\to\bigwedge^rT_X$,
or equivalently a nonzero section
\[
0\neq s\in
H^0\left(
X,\bigwedge^rT_X\otimes L^{-1}
\right).
\]
Apply Proposition~4.3 with $m=r$ and a Cartier divisor representing
$L$. The numerical hypothesis is given by (2.3). We obtain 
$X\simeq\mathbb P^n_\Bbbk$ and 
$L\simeq\mathcal O_{\mathbb P^n_\Bbbk}(r)$.

\medskip
\noindent
\textit{Case 2: $r=n$.}
Taking determinants of $\mathcal E\hookrightarrow T_X$ gives a
nonzero morphism 
$L\to
\det T_X\simeq\mathcal O_X(-K_X)$. 
If $F$ is its zero divisor, then
\begin{equation}\label{eq:determinant-divisor}
-K_X\sim L+F,
\qquad
F\geq0.
\end{equation}

Suppose first that $F\neq0$. By Lemma~2.2, $-F$ is not nef. Since
$\NEbar(X)$ is generated by its extremal
rays, there is an extremal ray $R$ such that
$F\cdot\gamma_R>0$. Since $F\cdot\gamma_R$ is a positive integer,
(2.3) and \eqref{eq:determinant-divisor} give
\[
-K_X\cdot\gamma_R
=
L\cdot\gamma_R+F\cdot\gamma_R
\geq n+1.
\]

Suppose next that $F=0$. The determinant of the morphism
$\mathcal E\to T_X$ is then nowhere vanishing, so the morphism is an
isomorphism. In particular, $\mathcal E\simeq T_X$ and $T_X$ is
ample. Choose any extremal ray $R$ and an invariant curve
$C_R\simeq\mathbb P^1_\Bbbk$ representing $\gamma_R$. Write
\[
T_X|_{C_R}
\simeq
\mathcal O_{\mathbb P^1_\Bbbk}(b_1)\oplus\cdots\oplus
\mathcal O_{\mathbb P^1_\Bbbk}(b_n),
\qquad
b_i\geq1.
\]
The differential of the closed immersion $C_R\hookrightarrow X$
gives a nonzero injection 
$T_{C_R}\simeq\mathcal O_{\mathbb P^1_\Bbbk}(2)
\to T_X|_{C_R}$.  
If all $b_i$ were equal to $1$, this morphism would be zero. Hence
at least one $b_i$ is at least $2$, and
\[
-K_X\cdot\gamma_R
=
\deg(T_X|_{C_R})
=
\sum_{i=1}^n b_i
\geq n+1.
\]

Thus, in either case, there is an extremal ray $R$ such that
$-K_X\cdot\gamma_R\geq n+1$. Lemma~4.4 gives
$X\simeq\mathbb P^n_\Bbbk$.

It remains to determine $\mathcal E$. If $r=n$ and $F=0$, the given
morphism is already an isomorphism 
$\cE\simeq T_{\mathbb P^n_\Bbbk}$. 
We may therefore assume that either $r<n$, or that $r=n$ and
$F\neq0$.

In the first case, Proposition~4.3 gives
$L\simeq\mathcal O_{\mathbb P^n_\Bbbk}(r)$. In the second case, write 
$L\simeq\mathcal O_{\mathbb P^n_\Bbbk}(d)$ 
and 
$F\sim aH$, 
where $H$ is a hyperplane and $a\geq1$. Since $\mathcal E$ is ample
of rank $n$, its restriction to a line shows that $d\geq n$.
Equation \eqref{eq:determinant-divisor} gives 
$n+1=d+a$, 
and hence $d=n$ and $a=1$. Thus, in every remaining case, 
$L\simeq\mathcal O_{\mathbb P^n_\Bbbk}(r)$. 

For any line $\ell\subset\mathbb P^n_\Bbbk$, Grothendieck's splitting
theorem gives
\[
\mathcal E|_\ell
\simeq
\mathcal O_\ell(a_1)\oplus\cdots\oplus
\mathcal O_\ell(a_r).
\]
Since $\mathcal E$ is ample, $a_i\geq1$ for every $i$, while 
$\sum_{i=1}^r a_i
=
\deg(L|_\ell)
=
r$. 
Therefore $a_i=1$ for every $i$. Hence
$\mathcal E(-1)$ is trivial on every line in $\mathbb P^n_\Bbbk$, and
Lemma~\ref{lem:trivial-on-lines} yields
\[
\mathcal E\simeq
\mathcal O_{\mathbb P^n_\Bbbk}(1)^{\oplus r}.
\]
This completes the proof.
\end{proof}

\begin{remark}\label{rem:nonequivariant}
The toric part of the proof uses only the positivity of the
restrictions of $\cE$ to invariant rational curves and the
determinant map induced by
$\cE\hookrightarrow T_X$. No torus-equivariant structure on
$\cE$, or equivariance of the inclusion, is needed. Once
$X\simeq\mathbb P^n_\Bbbk$ is known, the structure of $\cE$ is
determined by its restrictions to lines.
\end{remark}

\section{Proof of Theorem~\ref{thm:beauville}}\label{sec:beauville}

We prove Theorem~\ref{thm:beauville} by the same method as Theorem~\ref{thm:main}. We first record the elementary calculation needed in the top-degree case. Although Lemma~\ref{lem:index-n} follows easily from \cite[Theorem~3.2\,(1)]{fujino-osaka} or \cite[Theorem~1.3]{fujinosato-notes2}, we include a direct proof for the reader's convenience. 

\begin{lemma}\label{lem:index-n}
Assume $n\geq2$, and let $A$ be an ample Cartier divisor on $X$. If $-K_X\sim nA$, then $n=2$ and
\[
(X,\cO_X(A))\simeq
(\PP^1_\Bbbk\times\PP^1_\Bbbk,\cO(1,1)).
\]
\end{lemma}

\begin{proof}
Choose an extremal ray $R$, with primitive relation \eqref{eq:primitive-relation}. Since $A\cdot\gamma_R$ is a positive integer,
\[
n\leq n(A\cdot\gamma_R)=-K_X\cdot\gamma_R
=h-\sum_{j=1}^k a_j\leq h\leq n+1.
\]
As $n\geq2$, this anticanonical degree must be $n$.

Reid's description of an extremal primitive relation shows that
\[
\langle v_1,\ldots,\widehat v_i,\ldots,v_h,w_1,\ldots,w_k\rangle\in\Sigma
\qquad(1\leq i\leq h);
\]
see \cite[Theorem~1.5]{casagrande-contractible}. Thus $h+k-1\leq n$. If $k>0$, then
\[
n=h-\sum_{j=1}^k a_j\leq h-k\leq n+1-2k<n,
\]
a contradiction. Hence $k=0$, $h=n$, and the relation is 
$v_1+\cdots+v_n=0$. 
Since this relation is extremal, its contraction is a $\PP^{n-1}_\Bbbk$-bundle over a smooth complete toric curve, hence over $\PP^1_\Bbbk$; see \cite[Corollary~2.4]{casagrande-contractible}. This is the standard fan-theoretic description of the contraction and is valid in arbitrary characteristic.

Using the convention that $\PP(\mathcal V)$ parametrizes one-dimensional quotients of $\mathcal V$, Grothendieck's splitting theorem and a twist on the base give
\[
X\simeq
\PP_{\PP^1_\Bbbk}\left(\bigoplus_{i=1}^n\cO_{\PP^1_\Bbbk}(d_i)\right),
\qquad 0=d_1\leq d_2\leq\cdots\leq d_n.
\]
Let $H$ be the tautological divisor class, and let $G$ be the pullback of the class of a point on $\PP^1_\Bbbk$. The canonical divisor formula gives
\[
-K_X\sim nH+\left(2-\sum_{i=1}^n d_i\right)G.
\]
Let $C_0$ be the section corresponding to the quotient onto $\cO_{\PP^1_\Bbbk}(d_1)=\cO_{\PP^1_\Bbbk}$. Then $H\cdot C_0=0$ and $G\cdot C_0=1$, so
\[
n\leq n(A\cdot C_0)=-K_X\cdot C_0
=2-\sum_{i=1}^n d_i\leq2.
\]
Therefore $n=2$ and $d_1=d_2=0$, proving that $X\simeq\PP^1_\Bbbk\times\PP^1_\Bbbk$. Finally, $2A\sim-K_X$ and $\operatorname{Pic}(X)\simeq\ZZ^2$ give $\cO_X(A)\simeq\cO(1,1)$.
\end{proof}

\begin{proof}[Proof of Theorem~\ref{thm:beauville}]
If $n=1$, then $X\simeq\PP^1_\Bbbk$, $p=1$, and $\cL\simeq\cO_{\PP^1_\Bbbk}(d)$ for an integer $d>0$. The hypothesis says $H^0(\PP^1_\Bbbk,\cO_{\PP^1_\Bbbk}(2-d))\neq0$, so $d=1$ or $2$. We may therefore assume $n\geq2$. Choose an ample Cartier divisor $A$ with $\cL\simeq\cO_X(A)$.

\medskip
\noindent\emph{Case 1: $p<n$.}

Apply Proposition~\ref{prop:common-criterion} with $D=pA$ and $m=p$, using \eqref{eq:polarization-lower}. We obtain $X\simeq\PP^n_\Bbbk$ and $\cL^{\otimes p}\simeq\cO_{\PP^n_\Bbbk}(p)$. Since $\operatorname{Pic}(\PP^n_\Bbbk)\simeq\ZZ$, this gives $\cL\simeq\cO_{\PP^n_\Bbbk}(1)$.

\medskip
\noindent\emph{Case 2: $p=n$.}

The nonzero section is a section of $\cO_X(-K_X-nA)$. Its zero divisor $F$ satisfies
\begin{equation}\label{eq:beauville-zero-divisor}
-K_X\sim nA+F,\qquad F\geq0.
\end{equation}
If $F\neq0$, Lemma~\ref{lem:effective-antinef} shows that $-F$ is not nef. Since $\NEbar(X)$ is generated by its extremal rays, some extremal ray $R$ satisfies $F\cdot\gamma_R>0$. Integrality gives
\[
-K_X\cdot\gamma_R=n(A\cdot\gamma_R)+F\cdot\gamma_R\geq n+1.
\]
Lemma~\ref{lem:large-degree} yields $X\simeq\PP^n_\Bbbk$. Write $\cL\simeq\cO_{\PP^n_\Bbbk}(d)$, with $d\geq1$. Since $F$ is a nonzero effective divisor, \eqref{eq:beauville-zero-divisor} gives $n+1-nd>0$, and hence $d=1$.
If $F=0$, then $-K_X\sim nA$, and Lemma~\ref{lem:index-n} gives precisely case \textup{(iii)}.

For the converse, the Euler sequence on $\PP^n_\Bbbk$ shows that $T_{\PP^n_\Bbbk}(-1)$ is globally generated. Hence its nonzero exterior powers have nonzero global sections. In cases \textup{(ii)} and \textup{(iii)}, one has $p=n$ and $\bigwedge^nT_X\otimes\cL^{-n}\simeq\cO_X$.
\end{proof}

\begin{remark}[Dependence on the characteristic]\label{rem:characteristic-dependence}
For general torus-equivariant vector bundles, even the existence of a nonzero global section can depend on the characteristic. This occurs for the fibers of a single equivariant vector bundle on a smooth projective toric scheme over $\ZZ$, as the following example shows.

Let $Y=\PP^1_\Bbbk\times\PP^1_\Bbbk$, and let $\pi\colon S\to Y$ be the blow-up of its four torus-fixed points $(0,0)$, $(0,\infty)$, $(\infty,0)$, and $(\infty,\infty)$. Write $E_1,\ldots,E_4$ for the exceptional curves, set $E=E_1+\cdots+E_4$, and put $\mathcal V=T_S\otimes\cO_S(-E)$. Then
\begin{equation}\label{eq:characteristic-jump}
h^0(S,\mathcal V)=
\begin{cases}
0,&\operatorname{char}\Bbbk\neq2,\\
1,&\operatorname{char}\Bbbk=2.
\end{cases}
\end{equation}
Indeed, a section $s$ of $T_S(-E)$, viewed as a vector field on $S$, vanishes along $E$. Away from $E$, it gives a vector field on the complement of the four centers in $Y$. Since $T_Y$ is locally free and the omitted locus has codimension two, this vector field extends uniquely to $v\in H^0(Y,T_Y)$. The identity $d\pi(s)=\pi^*v$ shows that $v$ vanishes at all four centers. In the usual affine coordinates $x,y$, it therefore has the form $v=ax\partial_x+by\partial_y$, with $a,b\in\Bbbk$.

At $(0,0)$, the lift to the blow-up chart $y=xt$ is
\[
\widetilde v=ax\partial_x+(b-a)t\partial_t.
\]
Together with the other chart, this shows that the lift vanishes along the exceptional curve if and only if $a=b$. At the other centers, replacing $x$ or $y$ by its inverse changes the corresponding coefficient to its negative. Thus the four vanishing conditions are exactly $a-b=0$ and $a+b=0$. Conversely, these equations ensure that the lifts vanish along all four exceptional curves in both charts. Hence
\[
H^0(S,\mathcal V)\simeq
\{(a,b)\in\Bbbk^2\mid a-b=a+b=0\},
\]
which proves \eqref{eq:characteristic-jump}. In characteristic two, the lift of $x\partial_x+y\partial_y$ gives the nonzero section.

Both $S$ and $\mathcal V$ arise from the same construction over $\ZZ$, using the relative tangent bundle. Thus the jump does not result from choosing unrelated bundles in different characteristics. Finally, $E\cdot E_i=-1$, so $\cO_S(E)$ is not ample; this example does not satisfy the polarization hypothesis of Theorem~\ref{thm:beauville}.
\end{remark}

\begin{remark}[The toric hypothesis in positive characteristic]\label{rem:positive-characteristic-counterexamples}
The following example shows that the toric hypothesis cannot be omitted from either Theorem~\ref{thm:main} or Theorem~\ref{thm:beauville} in arbitrary characteristic. The example goes back to Wahl \cite[p.~316]{wahl-cohomological}; see also Furukawa \cite[Example~2.2, Proposition~2.5, and Remark~2.8]{furukawa-hyperquadrics}. Furukawa further proves that, among nonlinear smooth complete intersections of dimension at least two, the condition $H^0(X,T_X(-1))\neq0$ characterizes precisely the odd-dimensional quadrics in characteristic two \cite[Theorem~1.1]{furukawa-hyperquadrics}.

Suppose that $\operatorname{char}\Bbbk=2$, and let $d\geq3$ be odd. Consider
\[
Q=\{f=0\}\subset\PP^{d+1}_\Bbbk,
\qquad
f=z_0^2+z_1z_2+z_3z_4+\cdots+z_dz_{d+1}.
\]
The partial derivatives with respect to $z_1,\ldots,z_{d+1}$ do not vanish simultaneously on $Q$, so $Q$ is smooth. Via the Euler sequence restricted to $Q$ and twisted by $\cO_Q(-1)$, the constant vector $e_0=(1,0,\ldots,0)$ defines a section of $T_{\PP^{d+1}_\Bbbk}(-1)|_Q$. It is nowhere zero, since its only possible zero in $\PP^{d+1}_\Bbbk$ is $[e_0]$, which does not lie on $Q$. Moreover, $\partial f/\partial z_0=0$, so the twisted normal sequence
\[
0\longrightarrow T_Q(-1)
\longrightarrow T_{\PP^{d+1}_\Bbbk}(-1)|_Q
\xrightarrow{\,df\,}\cO_Q(1)\longrightarrow0
\]
shows that this section lies in $H^0(Q,T_Q(-1))$. Being nowhere vanishing, it yields an ample line subbundle $\cO_Q(1)\hookrightarrow T_Q$.

Since $Q\not\simeq\PP^d_\Bbbk$, this contradicts the conclusion of Theorem~\ref{thm:main} when the toric hypothesis is removed. Taking $\cL=\cO_Q(1)$ and $p=1<d$ also gives a counterexample to the positive-characteristic analogue of Beauville's characterization: its quadric alternative is allowed only in top exterior degree. Unlike the example in Remark~\ref{rem:characteristic-dependence}, the polarization here is ample, but the variety is not toric. Thus the two characterizations hold in the toric setting in every characteristic even though their general positive-characteristic analogues fail.
\end{remark}

\bibliographystyle{amsalpha}
\bibliography{ref}

\end{document}